\documentclass [12pt]{article}
\usepackage{amssymb}
\usepackage{amsmath}
\usepackage{epsfig}
\usepackage{graphicx}
\usepackage{amsthm}
\usepackage{amssymb}
\usepackage{amsfonts}
\usepackage{dirtytalk}
\allowdisplaybreaks
\usepackage{enumerate}
\usepackage{enumerate}

\newcommand{\HH}{\mathbb{H}}
\newtheorem{theorem}{Theorem}
[section]
\newtheorem{corollary}{Corollary}
[theorem]
\newtheorem{lemma}[theorem]{Lemma}

\begin{document}
\thispagestyle{empty}
\null\vspace{-1cm}
\medskip
\vspace{1.75cm}
\centerline{\textbf{{\ On the Right Eigenvalues of the Quaternionic Matrix Polynomials}}}
~~~~~~~~~~~~~~~~~~~~~~~~~~~~~~~~~~~~~~~~~~~~~~~~~~~~~~~~~~~~~~~~~~~~~~~~~~~~~~~~~~~~~~~~~~~~~~~~~~~~~~~~~~~~~~~~~~~~~~~~~~~~~~~~~~~~~~~~~~~~~~~~~~~~~~~~~~~

\centerline{\bf  {Ovaisa Jan} and  {Idrees Qasim}}
\centerline {Department of Mathematics, National Institute of Technology, Srinagar, India-190006}
\centerline { ovaisa\_2022phamth009@nitsri.ac.in, idreesf3@nitsri.ac.in}
~~~~~~~~~~~~~~~~~~~~~~~~~~~~~~~~~~~~~~~~~~~~~~~~~~~~~~~~~~~~~~~~~~~~~~~~~~~~~~~~~~~~~~~~~~~~~~~~~~~~~~~~~~~~~~~~~~~~~~~~~~~~~~~~~~~~~~~~~~~~~~~~~~~~~~~~~
\vskip0.1in
\textbf{Abstract:} This paper establishes new upper bounds for the right eigenvalues of monic matrix polynomials over the quaternion division algebra. The noncommutative nature of quaternion multiplication presents fundamental challenges in eigenvalue analysis, distinguishing this problem from the classical complex case. We use spectral norm inequalities for partitioned quaternionic matrices and apply them to quaternionic block companion matrices associated with monic matrix polynomials. By analyzing the structure of powers of these companion matrices we derive progressively sharper bounds for the right eigenvalues. Consequently, these bounds give bounds for the zeros of quaternionic polynomials. \\

	\vspace{0.10in}
	~~~~~~~~~~~~~~~~~~~~~~~~~~~~~~~~~~~~~~~~~~~~~~~~~~~~~~~~~~~~~~~~~~~~~~~~~~~~~~~~~~~~~~~~~~~~~~~~~~~~~~~~~~~~~~~~~~~~~~~~~~~~~~~~~~~~~~~~~~~~~~~~~~~~~~~~~
	
	\noindent {{\bf Mathematics Subject Classiﬁcation (2020):} 12E15, 34L15, 15A18, 15A66.}\\
	\vspace{0.1in}

\section{Introduction}
The problem of locating eigenvalues of matrix polynomials constitutes a fundamental area of matrix analysis with significant implications across scientific and engineering disciplines. Higham and Tisseur \cite{HT} pioneered the establishment of bounds for eigenvalues of complex matrix polynomials, initiating a rich line of inquiry that has been extensively developed \cite{BDN, BM, MCP, NVS}. The eigenvalues of the block companion matrix of a monic matrix polynomial and the eigenvalues of that matrix polynomial are same. So, there is a nice method to obtain the bounds for the absolute values of the eigenvalues of a matrix polynomial by finding the bounds for the absolute values of eigenvalues for the block companion matrix of a monic complex polynomial. The extension of such results to the quaternionic setting, however, encounters substantial obstacles due to the noncommutative nature of quaternion multiplication.\\
\indent The quaternion division algebra, discovered by Hamilton in 1844 \cite{WR}, provides a four-dimensional extension of the real numbers that preserves algebraic closure properties but sacrifices commutativity. This noncommutativity gives rise to two distinct types of eigenvalues left and right eigenvalues making the spectral analysis of quaternionic matrices and matrix polynomials significantly more intricate than their complex counterparts. Quaternionic matrices have been extensively studied with particular attention to eigenvalue existence and localization (for example, see \cite{AS, BA, BJL, LHC}). Similarly, the location of zeros of quaternionic polynomials has been extensively studied (for example, see\cite{QI, GVM, NBD}).\\
\indent Early work on bounding zeros of quaternionic polynomials was carried out by Opfer \cite{OG}. Some recent developments on the
location and computation of zeros of quaternionic polynomials can be found in \cite{AS, QI, NBD}. Subsequently Ahmad and Ali \cite{ ASSA} extended these ideas to matrix polynomials, providing bounds for eigenvalues via companion matrix constructions and norm estimates. Ali and Truhar \cite{INT} introduced localization regions for right eigenvalues of quaternionic matrix polynomials using Ostrowski-type theorems for block matrices and developed a right spectral radius inequality that yields both upper and lower bounds. Specifically, they established that for a monic quaternionic matrix polynomial
\[
\mathbf{L}(\lambda) = I\lambda^m + A_{m-1}\lambda^{m-1} + \cdots + A_1\lambda + A_0,
\]
all right eigenvalues lie within the disk
\[
\mathcal{K} = \left\{ z \in \mathbb{H} : |z| \leq \frac{1}{2}\left[1 + \|A_{m-1}\|_2 + \sqrt{\left(\|A_{m-1}\|_2 - 1\right)^2 + 4\sqrt{\sum_{i=0}^{m-2}\|A_i\|_2^2}} \right] \right\}=B_{1},
\]
as well as within another disk \(\mathcal{G}\) given by an analogous bound. Their approach also produced bounds for complex matrix polynomials and illustrated improvements over earlier results. More recently,  Jan and Qasim \cite{QJ} provided both upper and lower bounds for the absolute values of right eigenvalues of quaternionic matrix polynomials. Their work leverages subordinate norms of the coefficient matrices and provides quaternionic analogues of classical results, such as Cauchy's theorem.

In this paper, we establish new upper bounds for the right eigenvalues of monic matrix polynomials over the quaternion division algebra by using spectral norm inequalities for partitioned quaternionic matrices and applying them to quaternionic block companion matrices. Unlike previous approaches that analyze the companion matrix directly, we consider its higher powers to obtain progressively sharper bounds.

\section{Preliminary Knowledge.} Let $\mathbb{C}$ and $\mathbb{R}$ denote the fields of the complex and real numbers respectively. Let $\mathbb{H}$ be the set of quaternions defined as $$\mathbb{H}:=\{q=a+bi+cj+dk:a, b, c, d \in\mathbb{R} \},$$  
	 
	\noindent where $i^2=j^2=k^2=ijk=-1,~ij=k=-ji,~jk=i=-kj,~ki=j=-ik.$\\
For $q \in\mathbb{H}$,~ ${\bar{q}}=a-bi-cj-dk$ is the conjugate of $q$ and hence the modulus of a quaternion $q$ is given by $|q|=\sqrt{a^{2}+b^{2}+c^{2}+d^{2}}.$ The set of \(n\)-column vectors with entries in \(\mathbb{H}\) is denoted by \(\mathbb{H}^n\). The inner product for \(x, y \in \mathbb{H}^n\) is \(\langle x, y \rangle := y^H x\), and the norm is \(\|x\| := \sqrt{\langle x, x \rangle}\).
  ${M}_{m\times n}{(\mathbb{H})}$ denotes the space of ${m\times n}$ quaternion matrices and ${M}_{n}{(\mathbb{H})}$ denotes the space of ${n\times n}$ matrices over $\mathbb{H}.$ For $A=[a_{ij}] \in{M}_{n}{(\mathbb{H})}$ the transpose of $A$ is  \(A^T=[a_{ji}]\) and the conjugate transpose of $A$ is \(A^H=\overline{(A^T)}\).  Next we define norms on matrix $A.$
  For a square matrix $A\in M_{n}(\mathbb{H})$
  with enteries $a_{ij}\in\mathbb{H}$, the following norms are defined as \\ 
 1. The 1-Norm: $\|A\|_{1}:=\max\limits_{1\le j \le n}\sum\limits_{i=1}^{n}|a_{ij}|=\|A^{H}\|_{\infty}$\\   
 2. The $\infty$-Norm: $\|A\|_{\infty}:=\max\limits_{1\le i \le n}\sum\limits_{j=1}^{n}|a_{ij}|=\|A^{H}\|_{1}$\\
 3. The 2-Norm: $\|A||_{2}:=\sup\limits_{x\neq 0}\left\{ \frac{\|Ax\|_{2}}{\|x\|_{2}}:~x\in\mathbb{H}^{n}                         \right\} =\|A^{H}\|_{2}$\\
 4. The Frobenius Norm: $\|A\|_{F}:={(~\mbox{trace}~   A^{H}A)}^{1/2}.$\\ 
 
Since quaternion multiplication is non-commutative there exists two types of eigenvalues, namely left and right.\\

 \noindent\textbf{Definition 2.1:} Let $A\in M_{n}{(\mathbb{H})}$. Then the set of left eigenvalues of $A$ is defined as:\\
 $$\Lambda_{l}(A):=\{ \mu \in\mathbb{H} : Ax=\mu x ~\mbox{for some non-zero} ~x\in\mathbb{H}^{n} \}$$
 and the set of right eigenvalues of $A$ is defined as:
$$\Lambda_{r}(A):=\{ \mu \in\mathbb{H} : Ax=x \mu  ~\mbox{for some non-zero} ~x\in\mathbb{H}^{n} \}.$$

\noindent\textbf{Definition 2.2:} Let $A\in M_{n}{(\mathbb{H})}$. Then the left spectral radius of matrix $A$ is given as:\\  
$$ \rho_{l}(A) := \max \left\{|\mu|: \mu\in\Lambda_{l}(A)\right\}.$$
Similarly, the right spectral radius of a matrix $A$ is defined as:\\
$$ \rho_{r}(A) := \max \left\{|\mu|: \mu\in\Lambda_{r}(A)\right\}.$$

 \noindent

 Let \( \mathbb{P}_k(M_n(\mathbb{C})) \) be the space of matrix polynomials over the field of complex numbers. Then, \( \mathcal{P} \in \mathbb{P}_k(M_n(\mathbb{C})) \) is defined by  
\[\mathcal{P}(z) = \sum_{i=0}^k z^i A_i, \quad A_i \in M_n(\mathbb{C}), z \in \mathbb{C}.\]  
For some non zero \( x \in \mathbb{C}^n \), we have complex polynomial eigenvalue problem \( \mathcal{P}(z)x = 0 \). The polynomial \( \mathcal{P} \in \mathbb{P}_k(M_n(\mathbb{C})) \) is said to be regular if \( \det(\mathcal{P}(z)) \neq 0 \) for some \( z \in \mathbb{C} \). The set of eigenvalues of a regular polynomial \( \mathcal{P} \in \mathbb{P}_k(M_n(\mathbb{C})) \) is defined by  

\[\Lambda(\mathcal{P}) := \{ z \in \mathbb{C} : \det(\mathcal{P}(z)) = 0 \}.\]
 The above can be extended to the space of matrix polynomials over skew field of quaternions.\\

\noindent\textbf{Definition 2.3:} The space of right matrix polynomials of degree $(\leq k)$ over quaternion division algebra is denoted by \begin{equation}
    \mathbb{L}_{k}(M_n(\mathbb{H})) :=\left\{ L(\xi)=\sum\limits_{i=0}^{k} A_{i}\xi^{i} ,~ A_{i} \in M_{n}(\mathbb{H}),~ 0\leq i \leq k \right\}.\label{eq:01}
\end{equation}
where  $\xi$ commutes with quaternionic coefficients of matrix polynomial.\\
 \noindent Now, we define the right eigenvalues of $L\in\mathbb{L}_{k}(M_n(\mathbb{H})).$\\
 
 \noindent\textbf{Definition 2.4:} $\lambda\in\mathbb{H}$ is called a right eigenvalue of the polynomial $L \in\mathbb{L}_{k}(M_n(\mathbb{H}))$  if $A_0x+A_{1}x \lambda +A_{2}x \lambda^{2} +\cdots+A_{k}x\lambda^{k}=0$ for some non-zero $x\in\mathbb{H}^{n},$ where $x$ is the right eigenvector corresponding to right eigenvalue $\lambda.$ The set of right eigenvalues of the matrix polynomial $L$ is called the right spectrum of the polynomial $L$, denoted by $\Lambda_{r}{(L)}.$
\section{Bounds on the Right Eigenvalues of Quaternionic Matrix Polynomials.}
To find the bounds on right eigenvalues of quaternionic matrix polynomials. Assume $A_k = I_n$ in (\ref{eq:01}), we obtain the quaternionic monic matrix polynomial ${L}(\mu)$ as follows:

\begin{equation}
    {L}(\xi) := I\xi^k + A_{k-1}\xi^{k-1} + \cdots + A_1\xi + A_0, \label{eq:02}
\end{equation}
where $A_i \in M_n(\HH)$, $0 \leq i \leq k-1$, and $\xi$ commutes with quaternionic matrices. The right eigenvalues of ${L}(\mu)$ are given by the right eigenvalues of the $kn \times kn$ quaternionic companion matrix

\[
C_L := 
\begin{bmatrix}
0 & I_n & 0 & \cdots & 0 \\
0 & 0 & I_n & \cdots & 0 \\
\vdots & \vdots & \ddots & \ddots & \vdots \\
0 & 0 & 0 & \cdots & I_n \\
-A_0 & -A_1 & -A_2 & \cdots & -A_{k-1}
\end{bmatrix}.
\]
Next, we define the quaternionic block  matrix of $C_L$ as follows:

\[
C_B = \begin{bmatrix} C_{11} & C_{12} \\ C_{21} & C_{22} \end{bmatrix},
\]

where

\[
C_{11} = 
\begin{bmatrix}
0 & I_n & 0 & \cdots & 0 \\
0 & 0 & I_n & \cdots & 0 \\
\vdots & \vdots & \ddots & \vdots & \vdots \\
0 & 0 & 0 & \cdots & I_n \\
0 & 0 & 0 & \cdots & 0
\end{bmatrix}, \quad C_{12} = 
\begin{bmatrix}
0 \\
\vdots \\
0\\
I_n
\end{bmatrix},
\]

\[
C_{21} = 
\begin{bmatrix}
-A_0 & -A_1 & \cdots & -A_{k-2}
\end{bmatrix}, \quad C_{22} = -A_{k-1}.
\]
To find sharper bounds for right eigenvalues of ${L}(\xi)$, we take the power of the companion matrix $C_L$ which is as follows:

\[
C_L^2 :=
\begin{bmatrix}
0 & 0 & I_n & \cdots & 0 & 0 \\
\vdots & \vdots & \vdots & \dots & \vdots &\vdots\\
0 & 0 & 0 & \cdots & I_n & 0\\
0 & 0 & 0 & \cdots & 0& I_{n} \\
-A_0 & -A_1 & -A_2 & \cdots & -A_{k-2} & -A_{k-1}\\
B_0 & B_1 & B_2 & \cdots & B_{k-2} & B_{k-1}
\end{bmatrix},
\]
\[
C_L^3 :=
\begin{bmatrix}
0 & 0 &0& I_n & \cdots & 0 & 0 \\
\vdots & \vdots&\vdots & \vdots & \dots & \vdots &\vdots\\
0 & 0 & 0 &0& \cdots &  I_{n} & 0\\
0 & 0 & 0 & 0&\cdots & 0& I_{n} \\
-A_0 & -A_1 & -A_2 &-A_{3}& \cdots & -A_{k-2} & -A_{k-1}\\
B_0 & B_1 & B_2 & B_{3}&\cdots & B_{k-2} & B_{k-1}\\
C_0 & C_1 & C_2 & C_{3}&\cdots & C_{k-2} & C_{k-1}
\end{bmatrix},
\]
where $B_i = A_{k-1}A_i - A_{i-1}$, $0 \leq i \leq k-1,~ C_{i}=-A_{k-1}B_{i}+A_{k-2}A_{i}-A_{i-2},~$with~ $A_{-1}=A_{-2} = 0$.\\
Define the quaternionic block  matrix of $C_L^{2}$ and $C_L^{3}$ as follows:

\[
C_B^{2} = \begin{bmatrix} M_{11} & M_{12} \\ M_{21} & M_{22} \end{bmatrix},
\quad
C_B^{3} = \begin{bmatrix} K_{11} & K_{12} \\ K_{21} & K_{22} \end{bmatrix},
\]
where
\[
M_{11} = 
\begin{bmatrix}
0 & 0 & I_n & \cdots & 0 \\
\vdots & \vdots  & \vdots &  &\vdots\\
0 & 0 & 0 & \cdots & I_n \\
0 & 0 & 0 & \cdots & 0 \\
-A_0 & -A_1 & -A_2 & \cdots & -A_{k-2} \\
\end{bmatrix}, \quad M_{12} = 
\begin{bmatrix}
0 \\
\vdots \\
0\\
I_n\\ -A_{k-1}
\end{bmatrix},
\]

\[
M_{21} = 
\begin{bmatrix}
B_0 & B_1 & \cdots & B_{k-2}
\end{bmatrix}, \quad M_{22} = B_{k-1}
\] and
\[
K_{11} = 
\begin{bmatrix}
0 & 0 & 0& I_n & \cdots & 0 \\
\vdots & \vdots& \vdots  & \vdots &  &\vdots\\
0 & 0 & 0 & 0&\cdots & I_n \\
0 & 0 & 0 & 0&\cdots & 0 \\
0 & 0 & 0 & 0&\cdots & 0 \\
-A_0 & -A_1 & -A_2 & -A_{3}&\cdots & -A_{k-3} \\
\end{bmatrix}, \quad K_{12} = 
\begin{bmatrix}
0 & 0\\
\vdots& \vdots \\
0&0\\
I_{n}&0\\
0&I_{n}\\
-A_{k-2}&-A_{k-1}

\end{bmatrix},
\]

\[
K_{21} = 
\begin{bmatrix}
B_0 & B_1 & \cdots & B_{k-3}\\
C_0 & C_1 & \cdots & C_{k-3}
\end{bmatrix},  \quad K_{22} = 
\begin{bmatrix}
B_{k-2}&B_{k-1} \\
C_{k-2}&C_{k-1}\\ 
\end{bmatrix}.
\]

\section{Auxilliary Results}
\noindent For the proof of the theorems we need the following lemmas. The following lemma is mentioned in \cite{ASSA}
\begin{lemma}{\label{lemma 3.1}}
    Let $A= [a_{ij}]\in M_{n}(\mathbb{H)}.$ Then,  $\|A\|^2_{2}=\|A^{H}\|^{2}_{2}=\|A^{H}A\|_{2}=\|AA^{H}\|_{2}.$
\end{lemma}

\noindent The next lemma can be found in \cite{INT}.
    \begin{lemma}{\label{lemma 3.2}}
    Let \( A = [a_{ij}] \in M_n(\mathbb{H}) \) be partitioned as
    \[
    A = \begin{bmatrix}
        A_{11} & A_{12} \\
        A_{21} & A_{22}
    \end{bmatrix},
    \]
    where \( A_{ij} \in M_{n_i \times n_j}(\mathbb{H}) \) is the \((i,j)\) block of \( A \) such that \( n_1 + n_2 = n \), where \( i,j \in \{1,2\} \). If
    \[
    \tilde{A} = \begin{bmatrix}
        \|A_{11}\|_2 & \|A_{12}\|_2 \\
        \|A_{21}\|_2 & \|A_{22}\|_2
    \end{bmatrix},
    \]
    then the following inequalities hold:
    \begin{itemize}
        \item \( \rho_r(A) \leq \rho_r(\tilde{A}) \).
        \item \( \|A\|_2 \leq \|\tilde{A}\|_2 \).
    \end{itemize}
\end{lemma}

\begin{lemma}{\label{lemma 3.3}}
    Let \[
S := 
\begin{bmatrix}
0 &0 & I_n & \cdots & 0&0 \\
0 & 0 &0& \cdots &0 & 0 \\
\vdots & \vdots & \vdots & \ddots & \vdots \\
0 & 0 & 0 & \cdots & I_n &0\\
0 & 0 & 0 & \cdots & 0 & I_{n}\\
0 & 0 & 0 & \cdots & 0 &0\\
-A_0 & -A_1 & -A_2 & \cdots & -A_{k-3}&-A_{k-2}
\end{bmatrix}.
\]
Then, $$\|S\|^{2}_{2}\le \dfrac{1}{2}\left( 1+ \xi_0+ \sqrt{(1+\xi_{0})^{2}-4(\|A_{0}\|_{2}^{2}+\|A_{1}\|_{2}^{2}})\right),$$
where $\xi_{0}=\sum_{i=0}^{k-2}\|A_{i}\|_{2}^{2}.$
\end{lemma}
\begin{proof} By the definition of matrix 2- norm we have $\|S\|_{2}^{2}=\rho_{r}(SS^{H})$, so to find $\|S\|^{2}_{2}$, we compute

     \[
SS^{H}:= 
\begin{bmatrix}
I_{n} &0 & 0 & \cdots & 0& 0& -{A_{2}}^{H}\\
0 & I_{n} &0& \cdots &0 &0& {-A_{3}}^{H} \\
\vdots & \vdots & \vdots & \ddots & \vdots \\
0 & 0 & 0 & \cdots & I_n &0& {-A_{k-2}}^{H}\\
0 & 0 & 0 & \cdots & 0 &0 &0\\
-A_2 & -A_3 & -A_4 & \cdots & -A_{k-2}&0& \sum_{i=0}^{k-2} A_{i}A_{i}^{H}
\end{bmatrix}.
\] We write $SS^{H}=A$ as: 
\[
    A = \begin{bmatrix}
        S_{11} & S_{12} \\
        S_{21} & S_{22}
    \end{bmatrix},
    \]
where \[
S_{11} := 
\begin{bmatrix}
I_{n} &0 & 0 & \cdots & 0&0\\
0 & I_{n} &0& \cdots &0 &0  \\
\vdots & \vdots & \vdots & \ddots & \vdots&\vdots \\
0 & 0 & 0 & \cdots & I_n&0 \\
0 & 0 & 0 & \cdots & 0 &0\\
\end{bmatrix},
\quad S_{12} := 
\begin{bmatrix}
-{A_{2}}^{H}\\
{-A_{3}}^{H} \\
\vdots \\ {-A_{k-2}}^{H}\\0\\
\end{bmatrix}. \]          
\[S_{21} := \begin{bmatrix}
    -A_2 &-A_3&\cdots& -A_{k-2} &0
\end{bmatrix},\quad 
S_{22}=\sum_{i=0}^{k-2}A_{i}A_{i}^{H}.\]
Define
\[
\tilde{A} = \begin{bmatrix}1 & \|S_{12}\|_{2} \\ \|S_{21}\|_2 & \|S_{22}\|_2 \end{bmatrix}.
\]
By direct calculation, we have the following values.\\ 
$$\|S_{12}\|_{2}=\|S_{21}\|_{2}\le \left(\sum_{i=2}^{k-2}\|A_{i}\|_{2}^{2}\right)^\frac{1}{2}$$ and $$\|S_{22}\|_{2} \le  \sum_{i=0}^{k-2}\|A_{i}\|_{2}^{2}.$$\\
By applying Lemma \ref{lemma 3.2}, we obtain
\newpage
\begin{align*}
\rho_r(A) &\leq \rho_r(\tilde{A}) \\
&= \begin{aligned}[t] &\frac{1}{2} \Bigl[ \|S_{11}\|_{2} + \|S_{22}\|_2 \\
&\qquad + \sqrt{\bigl( \|S_{11}\|_{2} + \|S_{22}\|_2 \bigr)^2 - 4 \bigl( \|S_{11}\|_{2}\|S_{22}\|_{2}-\|S_{12}\|_{2}\|S_{21}\|_{2} \bigr) } \Bigr] \end{aligned} \\
&\le \begin{aligned}[t] &\frac{1}{2} \Bigl[ 1 + \sum_{i=0}^{k-2}\|A_{i}\|_{2}^{2} \\
&\qquad + \sqrt{\bigl( 1 + \sum_{i=0}^{k-2}\|A_{i}\|_{2}^{2}\bigr)^{2} - 4 \Bigl(\sum_{i=0}^{k-2}\|A_{i}\|_{2}^{2}-\sum_{i=2}^{k-2}\|A_{i}\|^{2}_{2} \Bigr)} \Bigr] \end{aligned} \\
&=\frac{1}{2} \Bigl[ 1 + \sum_{i=0}^{k-2}\|A_{i}\|_{2}^{2} + \sqrt{\bigl( 1 + \sum_{i=0}^{k-2}\|A_{i}\|_{2}^{2}\bigr)^{2} - 4 \bigl( \|A_{0}\|^{2}_{2}+\|A_{1}\|^{2}_{2}\bigr)} \Bigr].
\end{align*}
\end{proof}
\noindent By an analysis similiar to that used in the proof of Lemma \ref{lemma 3.3}, we have the following result. 
\begin{lemma}{\label{lemma 3.4}}
    Let \[
N := 
\begin{bmatrix}
0 &0 & 0&I_n & \cdots & 0&0 \\
0 & 0 &0&0& \cdots &0 & 0 \\
\vdots & \vdots & \vdots&\vdots & \ddots & \vdots & \vdots\\
0 & 0 & 0 & 0&\cdots & 0 &I_{n}\\
0 & 0 & 0 &0& \cdots & 0 & 0\\
0 & 0 & 0 &0& \cdots & 0 &0\\
-A_0 & -A_1 & -A_2 &-A_3& \cdots & -A_{k-4}&-A_{k-3}
\end{bmatrix}.
\]
Then $$\|N\|^{2}_{2}\le \dfrac{1}{2}\left( 1+ \alpha+ \sqrt{(1+\alpha)^{2}-4\left(\|A_{0}\|_{2}^{2}+\|A_{1}\|_{2}^{2}+\|A_{2}\|^{2}_{2}\right)}\right),$$
where $\alpha=\sum_{i=0}^{k-3}\|A_{i}\|_{2}^{2}.$
\end{lemma}

\section{Main Results}
In this section we state and prove our main results.
\begin{theorem}{\label{theorem 3.5}}
    For every right eigenvalue $\xi$ of the quaternionic monic matrix polynomial $L(\xi)$ the following inequality holds:\\
    \begin{equation*}
|\xi| \le \left( \frac{1}{2} \left[  \xi_{1}+\|B_{k-1}\|_{2}  +\sqrt{(\xi_{1}-\|B_{k-1}\|_{2})^{2}+4\xi_{2} \sqrt{1+\|A_{k-1}\|^{2}_{2}}} \right] \right)^{\frac{1}{2}},
\end{equation*}\\
where \begin{equation*}
\xi_{1} = \left( \frac{1}{2} \left[ 1+\sum_{i=0}^{k-2}\|A_{i}\|_{2}^{2}  +  \sqrt{\left(1+\sum_{i=0}^{k-2}\|A_{i}\|_{2}^{2}\right)^{2}-4\left(\|A_{0}\|_{2}^{2}+\|A_{1}\|_{2}^{2}\right)} ~~\right] \right)^{\frac{1}{2}} 
\end{equation*}
and 

\begin{equation*}
    \xi_{2}=\left( \sum_{i=0}^{k-2} \|B_{i}\|_{2}^{2}\right)^{\frac{1}{2}}.
\end{equation*}
\end{theorem}
\begin{proof}
   Let
\[
C^{2}_{B} = \begin{bmatrix} M_{11} & M_{12} \\ M_{21} & M_{22} \end{bmatrix}.
\] be the block  matrix of $C^{2}_{L}$, where \[
M_{11} = 
\begin{bmatrix}
0 & 0 & I_n & \cdots & 0 \\
\vdots & \vdots  & \vdots &  &\vdots\\
0 & 0 & 0 & \cdots & I_n \\
0 & 0 & 0 & \cdots & 0 \\
-A_0 & -A_1 & -A_2 & \cdots & -A_{k-2} \\
\end{bmatrix}, \quad M_{12} = 
\begin{bmatrix}
0 \\
\vdots \\
0\\
I_n\\ -A_{k-1}
\end{bmatrix},
\]

\[
M_{21} = 
\begin{bmatrix}
B_0 & B_1 & \cdots & B_{k-2}
\end{bmatrix}, \quad M_{22} = B_{k-1},
\] such that $B_i = A_{k-1}A_i - A_{i-1}$, $0 \leq i \leq k-1~ \mbox~{with}~ A_{-1} = 0$.
Define
\[
\tilde{C}^{2}_{B} = \begin{bmatrix} \|M_{11}\|_{2} & \|M_{12}\|_{2} \\ \|M_{21}\|_2 & \|B_{k-1}\|_2 \end{bmatrix}.
\]
From Lemma \ref{lemma 3.2}, we have
$$\|C_{B}^{2}\|_{2} \le\|\tilde{C_{B}^{2}}\|_{2}$$\\
From Lemma \ref{lemma 3.3}, $\|M_{11}\|_{2}\le\xi_{1},$ where \begin{equation*}
\xi_{1} = \left( \frac{1}{2} \left[ 1+\sum_{i=0}^{k-2}\|A_{i}\|_{2}^{2}  +  \sqrt{\left(1+\sum_{i=0}^{k-2}\|A_{i}\|_{2}^{2}\right)^{2}-4\left(\|A_{0}\|_{2}^{2}+\|A_{1}\|_{2}^{2}\right)} ~~\right] \right)^{\frac{1}{2}}.
\end{equation*}
To find $\|M_{12}\|_{2},$ we write the matrix $M_{12}$ as $M_{12}=X+Y,$ where
\[X := 
\begin{bmatrix}
0\\
 0 \\
\vdots \\ I_{n}\\0\\
\end{bmatrix}~~ \mbox{and} \quad Y := 
\begin{bmatrix}
0\\
 0 \\
\vdots \\ 0\\-A_{k-1}\\
\end{bmatrix}. \] 
By direct calculation, we have the following values.\\
$\|X^{H}Y\|_{2}=\|Y^{H}X\|_{2}=0,~\|X^{H}X\|_{2}=1~\mbox{and}~\|Y^{H}Y\|_{2}=\|A_{k-1}\|^{2}_{2}.$\\
Thus,\begin{align*}
    \|M_{12}\|^{2}_{2}=\|M^{H}_{12}M_{12}\|_{2}&=\|(X+Y)^{H}(X+Y)\|_{2}\\&=\|X^{H}Y+Y^{H}X+X^{H}X+Y^{H}Y\|_{2}\\&\le\left( \|X^{H}X\|_{2}+\|Y^{H}Y\|_{2}\right),
\end{align*}
so that
\begin{align*}
    \|M_{12}\|_{2}&\le \left( \|X^{H}X\|_{2}+\|Y^{H}Y\|_{2}\right)^\frac{1}{2}\\&=\left( 1+\|A_{k-1}\|^{2}_{2}\right)^\frac{1}{2}.
\end{align*}
It is easy to obtain $\|M_{21}\|_{2}\le\xi_{2},$ where \begin{equation*}
    \xi_{2}=\left( \sum_{i=0}^{k-2} \|B_{i}\|^{2}\right)^{\frac{1}{2}}.
\end{equation*}
From the above inequalities, we can easily write that $\tilde{C^{2}_{B}}\le B,$ where \[
B = \begin{bmatrix} \xi_{1} & \left( 1+\|A_{k-1}\|^{2}_{2}\right)^\frac{1}{2} \\ \xi_{2}& \|B_{k-1}\|_{2} \end{bmatrix}.
\]
Hence, we obtain $
    \rho_{r}(\tilde{C^{2}_{B}})\le \rho_{r}(B).$
Now from this inequality and Lemma \ref{lemma 3.2}, we have 
\begin{align*}
    \rho_r(C^{2}_{B}) \leq \rho_r(\tilde{C}^{2}_B) &\le \rho_{r}(B)\\&= \frac{1}{2} \left[ \xi_{1} + \|B_{k-1}\|_2 + \sqrt{\left( \xi_{1} - \|B_{k-1}\|_2  \right)^2 + 4 \xi_{2}\left( 1+\|A_{k-1}\|^{2}_{2}\right)^\frac{1}{2} }
\right],
\end{align*}
Consequently, we obtain \begin{equation*}
 \rho_r(C_{L})\le \left( \frac{1}{2} \left[ \xi_{1}+\|B_{k-1}\|_{2}  +\sqrt{(\xi_{1}-\|B_{k-1}\|_{2})^{2}+4\xi_{2} \sqrt{1+\|A_{k-1}\|^{2}_{2}}} \right] \right)^{\frac{1}{2}}
\end{equation*}
\end{proof}
\noindent If we take $n=1$ in (\ref{eq:02}), the matrix polynomial $L(\mu)$ reduces into quaternionic polynomial, say
 \begin{equation}
     p(\mu)=\mu^{k}+a_{k-1}\mu^{k-1}+\dots+a_{1}\mu+a_{0}\label{eq:03}
 \end{equation}
  we can write $b_{i}=a_{k-1}a_{i}-a_{i-1},~0\le i \le k-1, $ with $a_{-1}=0.$ Then we can easily see that $\|A_{k-1}\|^{2}_{2}=|a_{k-1}|^{2},~ \|B_{k-1}\|_{2}=|b_{k-1}|, ~ \sum_{i=0}^{k-2}\|B_{i}\|^{2}_{2}=\sum_{i=0}^{k-2}|b_{i}|_{2}^{2}$ and $ \sum_{i=0}^{k-2}\|A_{i}\|^{2}_{2}=\sum_{i=0}^{k-2}|a_{i}|_{2}^{2}.$\\
   Now, Theorem \ref{theorem 3.5} gives the following corollary for bounds on the zeros of the quaternionic polynomial.

\begin{corollary}
    If $z$ is any zero of quaternionic polynomial (\ref{eq:03}) of degree $n\ge 4,$ then 
    \begin{equation*}
|z| \le \left( \frac{1}{2} \left[  \xi_{1}+|b_{k-1}|  + \sqrt{(\xi_{1}-|b_{k-1}|)^{2}+4\xi_{2} \sqrt{1+|a_{k-1}|^{2}}} \right] \right)^{\frac{1}{2}},
\end{equation*}\\
where \begin{equation*}
\xi_{1} = \left( \frac{1}{2} \left[ 1+\sum_{i=0}^{k-2}|a_{i}|^{2}  +  \sqrt{\left(1+\sum_{i=0}^{k-2}|a_{i}|^{2}\right)^{2}-4\left(|a_{0}|^{2}+|a_{1}|^{2}\right)} ~~\right] \right)^{\frac{1}{2}}
\end{equation*}
and 
$\xi_{2}=\left( \sum_{i=0}^{k-2}|b_{i}|^{2}\right)^\frac{1}{2}.$

\end{corollary}

\begin{theorem}{\label{theorem 3.6}}
     For every right eigenvalue $\xi$ of the quaternionic monic matrix polynomial $L(\xi)$ the following inequality holds:\\
    \begin{equation*}
|\xi| \le \left( 1+\frac{1}{2} \left[ \beta^{2}_{1} + \beta^{2}_{2} +\sqrt{(\beta^{2}_{1}+\beta^{2}_{2})^{2}+4\left(2\beta_{1}\beta_{2} + 1 \right) }\right] \right)^{\frac{1}{4}},
\end{equation*}\\
where \begin{multline*}
 \beta_{1}=\left( \frac{1}{2} \left[ \sum_{i=0}^{k-3}\left(\|A_{i}\|^{2}_{2}+\|B_{i}\|_{2}^{2}\right)+ \right. \right. \\
\left. \left.\sqrt{ \left( \sum_{i=0}^{k-3}\left(\|A_{i}\|^{2}_{2}-\|B_{i}\|_{2}^{2} \right)     \right)^{2}            +4\left|\left|\sum_{i=0}^{k-3}A_{i}{B_{i}^{H}}\right|\right|_{2}^{2} }~~\right]  \right)^{\frac{1}{2}},
\end{multline*}
\begin{multline*} 
\beta_{2}=\left( \frac{1}{2} \left[ \sum_{i=k-2}^{k-1}\left(\|A_{i}\|^{2}_{2}+\|B_{i}\|_{2}^{2}\right)+ \right. \right. \\
\left. \left. \sqrt{ \left( \sum_{i=k-2}^{k-1}\left(\|A_{i}\|^{2}_{2}-\|B_{i}\|_{2}^{2} \right) \right)^{2} +4\left\|\sum_{i=k-2}^{k-1}A_{i}{B_{i}^{H}}\right\|_{2}^{2} }~~\right] \right)^{\frac{1}{2}}. 
\end{multline*}

\end{theorem}
\begin{proof}
   Let
\[
C^{2}_{B} = \begin{bmatrix} N_{11} & N_{12} \\ N_{21} & N_{22} \end{bmatrix}
\] be the block matrix of $C^{2}_{L}$
where,
\[
N_{11} = 
\begin{bmatrix}
0 & 0 & I_n & \cdots & 0 \\
\vdots & \vdots  & \vdots &  &\vdots\\
0 & 0 & 0 & \cdots & I_n \\
0 & 0 & 0 & \cdots & 0 \\
0 & 0 & 0 & \cdots & 0 \\
\end{bmatrix}, \quad N_{12} = 
\begin{bmatrix}
0 &0 \\
\vdots&\vdots \\
I_n & 0\\ 0 & I_{n}
\end{bmatrix},
\]

\[
N_{21} = 
\begin{bmatrix}
-A_0 & -A_1 & \cdots & -A_{k-3}\\
B_0 & B_1 & \cdots & B_{k-3}
\end{bmatrix}, \quad N_{22} = \begin{bmatrix}
    -A_{k-2}&-A_{k-1}\\B_{k-2}&B_{k-1}\end{bmatrix}
\]
Since, $\|N_{11}\|_{2}=\|N_{12}\|_{2}=1.$
To find $\|N_{21}\|_{2},$ we compute  \[
A=N_{21}N_{21}^{H} = 
\begin{bmatrix}
\sum_{i=0}^{k-3} A_{i}A_{i}^{H}& -\sum_{i=0}^{k-3}A_{i}A_{i}^{H}  \\
 -\sum_{i=0}^{k-3} B_{i}A_{i}^{H}& \sum_{i=0}^{k-3} B_{i}B_{i}^{H}
\end{bmatrix}. \]
Thus, \[\tilde{A}= 
\begin{bmatrix}
\|\sum_{i=0}^{k-3} A_{i}A_{i}^{H}\|_{2}& -\|\sum_{i=0}^{k-3}A_{i}A_{i}^{H}\|_{2}  \\
 -\|\sum_{i=0}^{k-3} B_{i}A_{i}^{H}\|_{2}& \|\sum_{i=0}^{k-3} B_{i}B_{i}^{H}\|_{2}
\end{bmatrix}. \] By applying Lemma \ref{lemma 3.2}, we obtain 

\begin{align*}
\rho_{r}(A) &\le \rho_{r}(\tilde{A}) 
= \frac{1}{2}\Biggl[ \biggl( \Bigl\|\sum_{i=0}^{k-3} A_i A_i^{H}\Bigr\|_2 + \Bigl\|\sum_{i=0}^{k-3} B_i B_i^{H}\Bigr\|_2 \biggr) \\
&~~~~~\qquad + \sqrt{ \biggl( \Bigl\|\sum_{i=0}^{k-3} A_i A_i^{H}\Bigr\|_2 - \Bigl\|\sum_{i=0}^{k-3} B_i B_i^{H}\Bigr\|_2 \biggr)^2 + 4\Bigl\|\sum_{i=0}^{k-2} A_i B_i^{H}\Bigr\|_2 \Bigl\|\sum_{i=0}^{k-2} B_i A_i^{H}\Bigr\|_2 } \Biggr] \\
&~~~~~~~~~~~~\le \frac{1}{2}\Biggl[ \biggl( \sum_{i=0}^{k-3} \|A_i A_i^{H}\|_2 + \sum_{i=0}^{k-3} \|B_i B_i^{H}\|_2 \biggr) \\
&~~~~~\qquad + \sqrt{ \biggl( \sum_{i=0}^{k-3} \|A_i A_i^{H}\|_2 - \sum_{i=0}^{k-3} \|B_i B_i^{H}\|_2 \biggr)^2 + 4\Bigl\|\sum_{i=0}^{k-2} A_i B_i^{H}\Bigr\|_2 \Bigl\|\sum_{i=0}^{k-2} B_i A_i^{H}\Bigr\|_2 } \Biggr] \\
&~~~~~~~~~~~= \frac{1}{2}\Biggl[ \biggl( \sum_{i=0}^{k-3} \|A_i\|_2^2 + \sum_{i=0}^{k-3} \|B_i\|_2^2 \biggr) \\
&~~~~~~\qquad + \sqrt{ \biggl( \sum_{i=0}^{k-3} \|A_i\|_2^2 - \sum_{i=0}^{k-3} \|B_i\|_2^2 \biggr)^2 + 4\Bigl\|\sum_{i=0}^{k-2} A_i B_i^{H}\Bigr\|_2 \Bigl\|\sum_{i=0}^{k-2} A_i B_i^{H}\Bigr\|_2 } \Biggr] \\
&~~~~~~~~~~~= \frac{1}{2}\Biggl[ \biggl( \sum_{i=0}^{k-3} \|A_i\|_2^2 + \sum_{i=0}^{k-3} \|B_i\|_2^2 \biggr) \\
&~~~~~~\qquad + \sqrt{ \biggl( \sum_{i=0}^{k-3} \|A_i\|_2^2 - \sum_{i=0}^{k-3} \|B_i\|_2^2 \biggr)^2 + 4\Bigl\|\sum_{i=0}^{k-2} A_i B_i^{H}\Bigr\|_2^2 } \Biggr].
\end{align*}
 From the above inequalities, we can easily conclude 
     $\|N_{21}\|_{2}\le \beta_{1},$ \\
where  \begin{multline*}
 \beta_{1}=\left( \frac{1}{2} \left[ \sum_{i=0}^{k-3}\left(\|A_{i}\|^{2}_{2}+\|B_{i}\|_{2}^{2}\right)+ \right. \right. \\
\left. \left.\sqrt{ \left( \sum_{i=0}^{k-3}\left(\|A_{i}\|^{2}_{2}-\|B_{i}\|_{2}^{2} \right)     \right)^{2}            +4\left|\left|\sum_{i=0}^{k-3}A_{i}{B_{i}^{H}}\right|\right|_{2}^{2} }~~\right]  \right)^{\frac{1}{2}}.
\end{multline*}
Similarly, we obtain $\|N_{22}\|_{2}\le \beta_{2},$ where
\begin{multline*} 
\beta_{2}=\left( \frac{1}{2} \left[ \sum_{i=k-2}^{k-1}\left(\|A_{i}\|^{2}_{2}+\|B_{i}\|_{2}^{2}\right)+ \right. \right. \\
\left. \left. \sqrt{ \left( \sum_{i=k-2}^{k-1}\left(\|A_{i}\|^{2}_{2}-\|B_{i}\|_{2}^{2} \right) \right)^{2} +4\left\|\sum_{i=k-2}^{k-1}A_{i}{B_{i}^{H}}\right\|_{2}^{2} }~~\right] \right)^{\frac{1}{2}}. 
\end{multline*}
Now using Lemma {\ref{lemma 3.2}}, we have 

$$ \|C_{B}^{2}\|_2 \leq \|\tilde{C_{B}^{2}}\|_2. $$
This implies
\begin{align*}
   \|C_{B}^{2}\|^2_{2} &\le \left\| \begin{bmatrix} \|N_{11}\|_{2} &\|N_{12}\|_{2}\\ \|N_{21}\|_{2} & \|N_{22}\|_{2}\end{bmatrix} \right\|^2_{2} \\&=\left\| \begin{bmatrix} 1 &1\\ \|N_{21}\|_{2} & \|N_{22}\|_{2} \end{bmatrix} \right\|^2_{2}\\&\le \left\|\begin{bmatrix} 1 &1\\ \beta_{1} & \beta_{2} \end{bmatrix}\right\|^2_{2}.
\end{align*}
Thus, 
$$ \|C_{B}^{2}\|_{2} \le \left( 1+\frac{1}{2} \left[ \beta^{2}_{1} + \beta^{2}_{2} +\sqrt{(\beta^{2}_{1}+\beta^{2}_{2})^{2}+4\left(2\beta_{1}\beta_{2} + 1 \right) }\right] \right)^{\frac{1}{2}}.$$

\noindent Hence every right eigenvalues satisfies $$|\xi| \le \left( 1+\frac{1}{2} \left[ \beta^{2}_{1} + \beta^{2}_{2} +\sqrt{(\beta^{2}_{1}
+\beta^{2}_{2})^{2}+4\left(2\beta_{1}\beta_{2} + 1 \right) }\right] \right)^{\frac{1}{4}},$$
\end{proof}
\noindent 
If we take $n=1$ in (\ref{eq:03}), we can write $b_i = a_{k-1}a_i - a_{i-1}$, $0 \leq i \leq k-1~$with~ $a_{-1} = 0$. Then we can easily see that $\|A_{i}\|_{2}=|a_{i}|~\mbox{and}~\|B_{i}\|_{2}=|b_{i}|. $ \\Now Theorem \ref{theorem 3.6} gives the following corollary for bounds on the zeros of the quaternionic polynomial.

  \begin{corollary}
      If $z$ is any zero of quaternionic polynomial (\ref{eq:03}) of degree $n\ge 4$, then\\
    \begin{equation*}
|z| \le \left( 1+\frac{1}{2} \left[ \beta^{2}_{1} + \beta^{2}_{2} +\sqrt{(\beta^{2}_{1}
+\beta^{2}_{2})^{2}+4\left(2\beta_{1}\beta_{2} + 1 \right) }\right] \right)^{\frac{1}{4}},
\end{equation*}\\
where \begin{equation*}
 \beta_{1}=\left( \frac{1}{2} \left[ \sum_{i=0}^{k-3}\left(|a_{i}|^{2}+|b_{i}|^{2}\right)+\sqrt{ \left( \sum_{i=0}^{k-3}\left(|a_{i}|^{2}-|b_{i}|^{2} \right)     \right)^{2}            +4\left|\sum_{i=0}^{k-3}a_{i}\overline{b_{i}}\right|^{2} }~~\right]  \right)^{\frac{1}{2}}
\end{equation*}
and
\begin{equation*}
    \beta_{2}=\left( \frac{1}{2} \left[ \sum_{i=k-2}^{k-1}\left(|a_{i}|^{2}+|b_{i}|^{2}\right)+\sqrt{ \left( \sum_{i=k-2}^{k-1}\left|a_{i}|^{2}-|b_{i}|^{2} \right)     \right)^{2}            +4\left|\sum_{i=k-2}^{k-1}a_{i}\overline{b_{i}}\right|^{2} }~~\right]  \right)^{\frac{1}{2}}.
\end{equation*}
  \end{corollary}

\begin{theorem}{\label{theorem 3.7}}
   For every right eigenvalue $\xi$ of the quaternionic monic matrix polynomial $L(\xi)$ the following inequality holds:\\
    \begin{equation*}
|\mu| \le \left( \frac{1}{2} \left[ \eta _{1} + \eta_{2} +\sqrt{(\eta_{1}-\eta_{2})^{2}+4\tau_{1}\tau_{2} } \right] \right)^{\frac{1}{3}},
\end{equation*}\\
where 
\begin{equation*}
    \eta_{1}=\left( \frac{1}{2} \left[ 1+ \sum_{i=0}^{k-3}\|A_{i}\|^{2}_{2} 
    +\sqrt{\left( 1+ \sum_{i=0}^{k-3}\|A_{i}\|^{2}_{2}\right)^{2} -4\left(\|A_{0}\|^{2}_{2} +\|A_{1}\|^{2}_{2}+\|A_{2}\|^{2}_{2}\right) }\right] \right)^\frac{1}{2},
\end{equation*}
\begin{equation*}
    \eta_{2}=\left( \frac{1}{2} \left[   \sum_{i=k-2}^{k-1} \left(\|B_{i}\|^{2}_{2}+\|C_{i}\|^{2}_{2} \right)    +\sqrt{\sum_{j=k-2}^{k-1} \left(\|B_{i}\|^{2}_{2}-\|C_{i}\|^{2}_{2} \right)^{2}+4\left|\left|\sum_{i=k-2}^{k-1}B_{i}{C_{i}^
    {H}} \right|\right|_{2}^{2}}~                   \right] \right)^{\frac{1}{2}},  
\end{equation*}

\begin{equation*}
\tau_{2}=\left( \frac{1}{2} \left[   \sum_{i=0}^{k-3} \left(\|B_{i}\|^{2}_{2}+\|C_{i}\|^{2}_{2} \right)    +\sqrt{ \left(\sum_{i=0}^{k-3}\|B_{i}\|^{2}_{2}-\|C_{i}\|^{2}_{2} \right)^{2}+4\left|\left|\sum_{i=0}^{k-3}B_{i}{C_{i}^{H}} \right|\right|_{2}^{2}}~                   \right] \right)^{\frac{1}{2}}, 
\end{equation*}
 and 
 \begin{equation*}
    \tau_{1}=\sqrt{\|A_{k-1}\|^{2}_{2}+\|A_{k-2}\|^{2}_{2}+1}.
\end{equation*}
\end{theorem}
\begin{proof}
 Let 
\[
C^{3}_{B} = \begin{bmatrix} K_{11} & K_{12} \\ K_{21} & K_{22} \end{bmatrix}.
\] be the block matrix of $C^{3}_{L}$
where,
\[
K_{11} = 
\begin{bmatrix}
0 & 0 & 0& I_n & \cdots & 0 \\
\vdots & \vdots& \vdots  & \vdots &  &\vdots\\
0 & 0 & 0 & 0&\cdots & I_n \\
0 & 0 & 0 & 0&\cdots & 0 \\
0 & 0 & 0 & 0&\cdots & 0 \\
-A_0 & -A_1 & -A_2 & -A_{3}&\cdots & -A_{k-3} \\
\end{bmatrix}, \quad K_{12} = 
\begin{bmatrix}
0 & 0\\
\vdots& \vdots \\
0&0\\
I_{n}&0\\
0&I_{n}\\
-A_{k-2}&-A_{k-1}

\end{bmatrix},
\]

\[
K_{21} = 
\begin{bmatrix}
B_0 & B_1 & \cdots & B_{k-3}\\
C_0 & C_1 & \cdots & C_{k-3}
\end{bmatrix},  \quad K_{22} = 
\begin{bmatrix}
B_{k-2}&B_{k-1} \\
C_{k-2}&C_{k-1}\\ 
\end{bmatrix}.
\]
Using Lemma \ref{lemma 3.2}, we have \\
\begin{align*}
    \rho_{r}(C_{B}^{3}) &\le \rho_{r}\left( \begin{bmatrix}
    \|K_{11}\|_{2} & \|K_{12}\|_{2}\\
    \|K_{21}\|_{2} & \|K_{22}\|_{2}
\end{bmatrix}\right)\\& = \frac{1}{2} \left[ \|K_{11}\|_{2} + \|K_{22}\|_2 + \sqrt{\left( \|K_{11}\|_{2} - \|K_{22}\|_2  \right)^2 + 4 \|K_{12}\|_{2}\|K_{21}\|_{2} }
\right].
\end{align*}
Using Lemma \ref{lemma 3.4}, we have $\|K_{11}\|_{2}=\eta_{1},$ 
where \begin{equation*}
    \eta_{1}=\left( \frac{1}{2} \left[ 1+ \sum_{i=0}^{k-3}\|A_{i}\|^{2}_{2} 
    +\sqrt{\left( 1+ \sum_{i=0}^{k-3}\|A_{i}\|^{2}_{2}\right)^{2} -4\left(\|A_{0}\|^{2}_{2} +\|A_{1}\|^{2}_{2}+\|A_{2}\|^{2}_{2}\right) }\right] \right)^\frac{1}{2}.
\end{equation*}
By tedious computation one can show that  $\|K_{12}\|_{2}=\tau_{1}, \|K_{21}\|_{2}=\tau_{2} $ and $\|K_{22}\|_{2}=\eta_{2},$ where
\begin{equation*}
    \eta_{2}=\left( \frac{1}{2} \left[   \sum_{i=k-2}^{k-1} \left(\|B_{i}\|^{2}_{2}+\|C_{i}\|^{2}_{2} \right)    +\sqrt{\sum_{j=k-2}^{k-1} \left(\|B_{i}\|^{2}_{2}-\|C_{i}\|^{2}_{2} \right)^{2}+4\left|\left|\sum_{i=k-2}^{k-1}B_{i}{C_{i}^
    {H}} \right|\right|_{2}^{2}}~                   \right] \right)^{\frac{1}{2}},  
\end{equation*}

\begin{equation*}
\tau_{2}=\left( \frac{1}{2} \left[   \sum_{i=0}^{k-3} \left(\|B_{i}\|^{2}_{2}+\|C_{i}\|^{2}_{2} \right)    +\sqrt{ \left(\sum_{i=0}^{k-3}\|B_{i}\|^{2}_{2}-\|C_{i}\|^{2}_{2} \right)^{2}+4\left|\left|\sum_{i=0}^{k-3}B_{i}{C_{i}^{H}} \right|\right|_{2}^{2}}~                   \right] \right)^{\frac{1}{2}}, 
\end{equation*}
 and 
 \begin{equation*}
    \tau_{1}=\sqrt{\|A_{k-1}\|^{2}_{2}+\|A_{k-2}\|^{2}_{2}+1}
\end{equation*}
It follows that 
\begin{equation*}
|\mu| \le \left( \frac{1}{2} \left[ \eta _{1} + \eta_{2} +\sqrt{(\eta_{1}-\eta_{2})^{2}+4\tau_{1}\tau_{2} } \right] \right)^{\frac{1}{3}}.
\end{equation*}
\end{proof}

\noindent For $n=1$ in (\ref{eq:03}), we can write $b_i = a_{k-1}a_i - a_{i-1}$, $0 \leq i \leq k-1,~ c_{i}=-a_{k-1}b_{i}+a_{k-2}a_{i}-a_{i-2},~$with~ $a_{-1}=a_{-2} = 0$. Then one can easily see that $\|A_{i}\|_{2}=|a_{i}|,~\|B_{i}\|_{2}=|b_{i}| ~\mbox{and}~\|C_{i}\|_{2}=|c_{i}|.$ \\Now Theorem \ref{theorem 3.7} gives the following corollary for bounds on the zeros of the quaternionic polynomial.
\begin{corollary}
    If $z$ is any zero of polynomial (\ref{eq:03}) of degree $n\ge 5,$ then 
    \begin{equation*}
|z| \le \left( \frac{1}{2} \left[ \eta _{1} + \eta_{2} +\sqrt{(\eta_{1}-\eta_{2})^{2}+4\tau_{1}\tau_{2} } \right] \right)^{\frac{1}{3}},
\end{equation*}
where  \begin{equation*}
    \eta_{1}=\left( \frac{1}{2} \left[   \sum_{i=k-2}^{k-1} \left(|b_{i}|^{2}+|c_{i}|^{2} \right)    +\sqrt{\sum_{i=k-2}^{k-1} \left(|b_{i}|^{2}-|c_{j}|^{2} \right)^{2}+4\left|\sum_{i=k-2}^{k-1}b_{i}\overline{c_{i}} \right|^{2}}~                   \right] \right)^{\frac{1}{2}},  
\end{equation*}
\begin{equation*}
    \eta_{2}=\left( \frac{1}{2} \left[ 1+ \sum_{i=0}^{k-3}|a_{i}|^{2} 
    +\sqrt{\left( 1+ \sum_{i=0}^{k-3}|a_{i}|^{2}\right)^{2} -4\left(|a_{0}|^{2} +|a_{1}|^{2}+|a_{2}|^{2}\right) }\right] \right)^\frac{1}{2},
\end{equation*}

\begin{equation*}
\tau_{2}=\left( \frac{1}{2} \left[   \sum_{i=0}^{k-3} \left(|b_{i}|^{2}+|c_{i}|^{2} \right)    +\sqrt{ \left(\sum_{i=0}^{k-3}|b_{i}|^{2}-|c_{i}|^{2} \right)^{2}+4\left|\sum_{i=0}^{k-3}b_{i}\overline{c_{i}} \right|^{2}}~                   \right] \right)^{\frac{1}{2}}, 
\end{equation*}
 and 
 \begin{equation*}
    \tau_{1}=\sqrt{|a_{k-1}|^{2}+|a_{k-2}|^{2}+1}.
\end{equation*}
\end{corollary}

\section{Conclusion}

This paper has established new upper bounds for the right eigenvalues of monic matrix polynomials over the quaternion division algebra. By employing spectral norm inequalities for partitioned quaternionic matrices and applying them to quaternionic block companion matrices, we have developed three main theorems that yield progressively sharper bounds for right eigenvalues. A key innovation in our approach lies in analyzing higher powers of the companion matrix rather than examining it directly, which enables the derivation of improved estimates. The results naturally specialize to bounds for zeros of scalar quaternionic polynomials, offering enhancements over existing bounds in the literature. Throughout this work, we have carefully addressed the fundamental challenges arising from the noncommutative nature of quaternion multiplication, demonstrating that classical matrix analysis techniques can be successfully extended to the quaternionic setting through appropriate modifications.
Several promising avenues for future research emerge from this work:

\section{Future Research Direction} This paper focuses on upper bounds for right eigenvalues, establishing complementary lower bounds would provide complete localization regions for eigenvalues of quaternionic matrix polynomials. Generalizing these results to non-monic quaternionic matrix polynomials would broaden the applicability of the bounds. Investigating eigenvalue bounds for quaternionic matrix polynomials with special structures (such as symmetric or Hermitian polynomials) could yield tighter estimates.

\section{Declaration}
\textbf{Availabilty of data and material}\\
 Data availability is not applicable to this article as no new data were created or analyzed in this study.\\

\noindent \textbf{Competing interests}\\
The author declare that they have no competing interests.\\

\noindent \textbf{Funding}\\
The research of first author is supported by DST Inspire Fellowship (ID No. IF210629).

\end{document}